\documentclass[10pt]{article}
   \usepackage{graphicx}
\usepackage{latexsym}
\usepackage{float}
\usepackage{amsmath}
\usepackage{cite}
\usepackage{amssymb}
\usepackage{amsthm}
\usepackage{url}
\usepackage{tabulary}
\usepackage{booktabs}
 \usepackage{helvet}

\theoremstyle{plain}
\newtheorem{theorem}{Theorem}[section]

\newtheorem{lemma}[theorem]{Lemma}
\newtheorem{corollary}[theorem]{Corollary}

\theoremstyle{definition}
\newtheorem*{definition}{Definition}
\newtheorem{Ex}[theorem]{Example}

\theoremstyle{remark}

\def \Q {{\mathbb Q}}

\def \Z {{\mathbb Z}}

\def \K {{\mathbb K}}

\def \0 {{\mathbf 0}}

\date{}

\title{Irreducibility of integer-valued polynomials I}

\author{Devendra Prasad\\
   devendraprasad@iisertirupati.ac.in\\
   Department of Mathematics\\ IISER-Tirupati\\ Tirupati, Andhra Pradesh\\
 India, 517507\\  
  }
 
\begin{document}
\maketitle


\begin{abstract} Let $S \subset R$ be an arbitrary subset of a unique factorization domain $R$ and $\K$ be the field of fractions of $R$. The ring of integer-valued polynomials over $S$ is the set $\mathrm{Int}(S,R)= \{ f \in \K[x]: f(a) \in R\ \forall\ a \in S \}.$ This article is an effort to study the  irreducibility of integer-valued polynomials over arbitrary subsets of a unique  factorization domain. We   give a method to construct    special kinds of   sequences, which we  call  $d$-sequences. We then use  these sequences   to obtain a criteria for the irreducibility of the polynomials in $\mathrm{Int}(S,R).$ In some special cases, we explicitly construct these sequences and use these sequences to check the irreducibility of some polynomials in $\mathrm{Int}(S,R).$  At the end,  we suggest a generalization of our results to an arbitrary subset of a Dedekind domain.
 
\end{abstract}

\section{Introduction}

For a given subset $S$ of a domain $R$ the set of polynomials
$$ \mathrm{Int}(S,R  )=    \{f \in \K[x]: f(S) \subset R \}  ,   $$
 where $\K$ is the   field of fractions of $R$, forms a   ring. This ring is termed as the ring of integer-valued polynomials over $S$. A general reference for this topic could be Cahen and Chabert  \cite{Cahen} and some interesting results on the topic can be found in  \cite{Whatyoushouldknowaboutintegervaluedpolynomials} \cite{Chabert} \cite{Cahenfactorial}   \cite{fdnnfrisch2} and \cite{fdnnwerner3}. This ring is very rich in properties and is helpful in constructing examples/counterexamples in commutative algebra. In the previous few decades this ring   attracted the attention of several  mathematicians  and now the study of this ring has become a major field of specialization. In the case when $S=R$, we just write $ \mathrm{Int}(R)  $ instead of $ \mathrm{Int}(R,R ) .$

 \medskip

 In ring theory, one of the most fascinating concepts is irreducibility. The irreducibility of  polynomials has a venerable history  but in the  case of the ring of integer-valued polynomials, irreducibility has not been explored that much. 
  Only some methods are known so far and they are only for particular rings. For the  interested   readers, we give a short summary of articles dealing with the irreducibility of integer-valued polynomials.

 \medskip

  In 2005, Chapman and McClain \cite{Chapman} gave a criteria for testing the irreducibility of polynomials in $ \mathrm{Int}(S,R  )$ where $R$ is a unique factorization domain.  Peruginelli \cite{Peruginelli'sapplication} gave a computational method to test the irreducibility of polynomials in    $ \mathrm{Int}(\Z)  $ for some special polynomials of $\Q[x]$.   Antoniou, Nakato and   Rissner \cite{snakirr} introduced `table method'   to check   the irreducibility of polynomials in    $ \mathrm{Int}(\Z).$   For a summary of work on the irreducibility of integer-valued polynomials we refer to Prasad, Rajkumar and Reddy  \cite{prasadsurvey}, where a whole section is devoted to the irreducibility of integer-valued polynomials.

 \medskip

 The organization of the paper is as follows. In section \ref{secp}, we present some preliminaries and fix   notations for the whole paper. In section \ref{secd}, we introduce $d$-sequences and give some examples. In   section \ref{secirr},  we 
 obtain a criteria
  for the irreducibility  of polynomials in $ \mathrm{Int}(S,R)$  in the case when $R$ is a unique factorization domain and $S$ is an arbitrary subset. We give some examples to explain how  sometimes $d$-sequences can be very helpful and viable in testing the irreducibility of polynomials in $\mathrm{Int}(S,R).$

  \medskip

   As per our knowledge there is no criteria  known till date to test the irreducibility of polynomials in $ \mathrm{Int}(S,R)$    when $R$ is a Dedekind domain and $S$ is an arbitrary subset of $R$. In   section \ref{secirrgen}, we suggest a generalization of our results to get a criteria in this case for the first time. Finally, we show  how sometimes our results remain valid for the ring of integer-valued polynomials over an arbitrary subset of a domain.
  
 \section{ Preliminaries and notations}\label{secp}
  
   We start this section by fixing a few notations. Throughout the article $R$ denotes a unique factorization domain (UFD) with the field of fractions $\K$ and $S$ denotes an arbitrary subset of $R$. For a  polynomial  $f \in \K[x],\ g$ denotes the unique polynomial in $R[x]$ such that $f= \tfrac{g}{d}$, where $d \in R$ is also unique. Recall that an element $u$ of a ring $A$  is   said to be a {\em  unit} if  we can find an element $v \in A$ such that $uv=1.$  
   A non-zero non-unit element  $\alpha $   of a ring  $A$ is said to be an {\em irreducible element} if it is not a product of two non-units. Equivalently, if $$\alpha =\alpha_1 \alpha_2 $$ for $\alpha_1 ,\alpha_2 \in A$, then either $\alpha_1$ is a unit or  $\alpha_2$ is a unit.
      For brevity, we just   call `irreducible' instead of an irreducible element, where the ring automatically comes from the context.

   \medskip

  Given a subset $S \subset R$ and a  polynomial  $f= \tfrac{g}{d} \in \mathrm{Int}(S,R  )   $, consider the following subset of $R$
   
   $$T_f =\{ f(a) = \tfrac{g(a)}{d}: a \in S \}.$$
   
  If each element of $T_f$ is a multiple of some non-unit element $d' \in R$, then the polynomial cannot be irreducible since we have the proper factorization
   
   $$f=d' . \dfrac{f}{d'}$$
    in  $\mathrm{Int}(S,R  ).  $ In order to test the irreducibility of a polynomial $f$, we must assume that    each element of $T_f$ is not a multiple of some non-unit element  of $ R.$ Such a polynomial is said to be {\em `image primitive' }. Throughout the article,   a polynomial in $\mathrm{Int}(S,R  )$ refers to an  image primitive polynomial in  $\mathrm{Int}(S,R  ).$ Also, for brevity, an irreducible  polynomial   refers  to an irreducible polynomial in Int$(S,R)$, where $S$ and $R$  automatically come from the context. We denote the highest power of a prime ideal $P$ dividing an ideal $I$ by $w_P(I).$ For instance, $w_2(12)=2^2.$

 \medskip

  \section{$d$-sequences} \label{secd}

     In this section, we construct    special kinds of sequences called  {\em $d$-sequences}. Before introducing these kinds of sequences we need the notion of {\em  $\pi$-sequences}.  We know that the ideal generated by some irreducible element $\pi \in R$ is always a prime ideal, hence the ring $R_{(\pi)}$ is a local ring.

 \medskip

 A given subset $S \subset R$  can also be seen as a subset of the local ring $R_{(\pi)}$  for any   prime ideal $(\pi) \subset R.$ With this assumption we give the definition of $\pi$-sequences.

     \begin{definition} A sequence $\{ u_i \}_{i \geq 0}$ of elements of  $S \subset R$ is said to be a $\pi$-sequence if   for each   $k>0$,    $u_k \in S$   satisfies  
 
 $$\tfrac{(x-u_0)\ldots (x- u_{k-1})}{(u_k-u_0) \ldots (u_k-u_{k-1})} \in \mathrm{Int}(S,R_{(\pi)} ).$$


\end{definition} 
  
  \medskip

   In this way we get a sequence of elements $\{ u_i \}_{i \geq 0}$ in $S$ with arbitrary $u_0$. These kinds of sequences were also studied by Bhargava \cite{Bhar1} in a slightly different way to construct his generalized factorials. With this definition in hand, we define  $d$-sequences as follows.  

 \medskip


    \begin{definition} For a given element $d \in R$,   
     let $\pi_1, \pi_2,\ldots, \pi_r$
 be all the   irreducibles   of $R$ dividing  $d$. Let for   $1 \leq j \leq r$, $\{ u_{ij} \}_{i \geq 0} $ be a $\pi_j$-sequence of $S$ and $\pi_j^{e_{kj} }$ be   $(u_{kj}-u_{0j})\ldots (u_{kj}- u_{k-1j})$ viewed as  a member of  the ring $R_{(\pi_j)}$. Then a $d$-sequence $ \{ x_i \}_{0 \leq i \leq k}$   of $S$  of length $k$  is a solution to the following congruences
 
   
   \begin{equation}\label{crt}
   x_i \equiv u_{ij} \mod \pi_j^{e_{kj}+1}\ \forall\ 1 \leq j \leq r,
   \end{equation}

  where  $0 \leq i \leq k.$ 
\end{definition}

   By Chinese remainder theorem we get infinitely many solutions of Eq. (\ref{crt}). We fix a solution of Eq. (\ref{crt}) for each $i$   and get a sequence $a_0, a_1, \ldots, a_k$ of $k+1$ elements. 
    Such a sequence may be inside $S$ or may not be. We just call `a $d$-sequence' if  the subset $S$ is clear from the context.   This sequence is important throughout our study. Before proceeding it is apropos to give a few examples of $d$-sequences.
  
  \begin{Ex} In the case when $R=\Z$ and $S =a \Z+b$  where $a,b \in \Z,$ the sequence $b,a+b,\ldots, ak+b$ is  a $d$-sequence of length $k$ for every $d \in \Z$. This is because the sequence $b,a+b,\ldots$ is  a $p$-sequence for every prime number $p$. 
  
  In fact, any $k +1$ consecutive terms of $S$ form a  $d$-sequence of length $k$ for every $d \in \Z$.

  \end{Ex} 
  
  \begin{Ex} In the case when $R=\Z$ and $S $ is the set of square numbers including zero, the first $k +1$ consecutive terms of $S$ starting from zero form a  $d$-sequence of length $k$ for every $d \in \Z$. This can be shown by the same reasoning as in the previous example.
  
  \end{Ex}

  Recall that, for a given subset $S \subset R,$ the {\em fixed divisor } of a polynomial $f \in R[x]$ over $S$ is the greatest common divisor of the values taken by $f$ over $S$. This quantity is denoted by $d(S,f).$ Thus,
  
  $$d(S,f)=  \gcd \{ f(a): a \in S    \} .$$
  
  Classically, this quantity was applied to the problems of the ring of integer-valued polynomials only but recently mathematicians used this quantity to generalize some number theoretic problems (see for  instance, \cite{Prasad2019}, \cite{vajaituapp} and \cite{vajaituapp1}) also.
 For some latest results on fixed divisors we refer to Semwal, Rajkumar and Reddy \cite{Devendra} and for a solid summary of literature on fixed divisors we highly recommend   Prasad, Rajkumar and Reddy \cite{prasadsurvey} (see \cite{Devendrafixed} also).

  \medskip

   A sequence of distinct elements $\{ a_i \}_{i \geq 0}$ of $S$ is said to be  a {\em  fixed divisor sequence} (see Prasad, Rajkumar and Reddy \cite{prasadsurvey}) if  for every $k >0,\ \exists\ l_k \in \Z,$  such that for every polynomial $f$ of degree $k$  $$d(S,f)=(f(a_0),f(a_1), \ldots, f(a_{l_k}))
,$$  
  and no proper subset of $\{a_0,a_1, \ldots, a_{l_k} \}$ 
  determines the fixed divisor of all the degree $k$ polynomials. For instance, the sequence $0,1,2, \ldots$ is a fixed divisor sequence in $\Z$ with $l_k=k\ \forall\ k >0.$

  \begin{Ex} Let $S $ be a subset of $R$ with a fixed divisor sequence $\{ a_i \}_{i \geq 0}$. If for every positive integer $k, l_k=k$ then $ a_0,a_1, \ldots, a_k$ is a $d$-sequence of length $k$ for every $d \in R$ since the sequence $ a_0,a_1, \ldots $ is a $\pi$-sequence for all irreducible $\pi \in R$.

  \end{Ex}
  
  In all the   examples given so far, $d$-sequences always belong to the set. Now we give an example where this is not the case.

  \begin{Ex} Let $S
   $ be the set of prime numbers in $\Z$ and we wish to construct a $6$-sequence of   length four.   
We have  $2,3,5,7,17$ and $2,3,7,5,19$ as    a $2$-sequence and a $3$-sequence respectively of length four. Now we have the following factorization
$$(17-2)(17-3)(17-5)(17-7)=2^4 .1575$$
and
$$(19-2)(19-3)(19-7)(19-5)=3^1.15232.$$
 
 The element 1575 is invertible in $\Z_{(2)}$  and so is 15232 in  $\Z_{(3)}.$ Hence, a first term of a $6$-sequence is a solution of the congruences
 $$x \equiv 2\ ( \mathrm{mod}\ 32)$$ and
$$x \equiv 2\ ( \mathrm{mod}\ 9).$$


A solution to the above congruence is $a_0=290.$ Similarly at the last (fifth) step we solve
the congruences
 $$x \equiv 17\ ( \mathrm{mod}\ 32)$$ and
$$x \equiv 19\ ( \mathrm{mod}\ 9),$$
 
to get a solution  $a_4=145.$
 The readers can compute the other terms   to get  $ 290,291, 133, 445, 145  $ as a $6$-sequence of length four in which all the elements are not members of $S$.
 \end{Ex}

  \medskip

  \section{Irreducibility of integer-valued polynomials} \label{secirr}
  
Before coming to the main result, we prove an important lemma which is helpful in proving our main result.
  \begin{lemma}\label{ivpcr} Let $a_0, a_1, \ldots, a_k$ be a $d$-sequence of length $k$ for some $d \in R$ and a given positive integer $k$. 
   Then, for any  polynomial $ f' = \tfrac{g'}{d'} \in \K[x],$ where $d' \mid d$,  of degree $k' \leq k$ the following holds
  
    $$f' \in \mathrm{Int}(S,R  )   \Leftrightarrow f'(a_i) \in R\ \forall\ 0 \leq i \leq k'.$$
  \end{lemma}
  
  \begin{proof} Observe that, for any $\pi \mid d$,  $w_{\pi}((a_i-a_0)(a_i-a_1) \ldots (a_i-a_{i-1})) = w_{\pi}((b_i-b_0)(b_i-b_1) \ldots (b_i-b_{i-1}))\ \forall\ 1 \leq i \leq k$, where $b_0,b_1, \ldots$ is a $\pi$-sequence in $S$. Consider   the following representation
  
  \begin{equation}\label{repre}
  f' = \tfrac{g'}{d'} = \sum_{i=0}^{k'} c_i  \tfrac{(x-a_0)(x-a_1) \ldots (x-a_{i-1})}{d'},
  \end{equation}
  
 where $c_i \in R\ \forall\ 0 \leq i \leq k'.$
 Now we have the following observation
 
 \begin{equation*} 
\begin{split}
f'(a_i) \in R\ \forall\ 0 \leq i \leq k' & \Leftrightarrow c_i\tfrac{(a_i-a_0)(a_i-a_1) \ldots (a_i-a_{i-1})}{d'} \in  R\ \forall\ 0 \leq i \leq k',\\
 & \Leftrightarrow c_i\tfrac{(b_i-b_0)(b_i-b_1) \ldots (b_i-b_{i-1})}{d'} \in R\   \forall\ 0 \leq i \leq k' .\\
  \end{split}
\end{equation*}
  
  By the definition of $\pi$-sequences the above holds iff
    for any arbitrary element \ $\alpha\ \in S, $  
    $$  \ c_i\tfrac{(\alpha-b_0)(\alpha-b_1) \ldots (\alpha-b_{i-1})}{d'} \in R\   \forall\ 0 \leq i \leq k',$$
 which is true iff  $$  \ c_i\tfrac{(\alpha-a_0)(\alpha-a_1) \ldots (\alpha-a_{i-1})}{d'} \in R\   \forall\ 0 \leq i \leq k',$$
 for any arbitrary element  $\alpha\ \in S.$  i.e., if and only if $f'(\alpha) \in R\ \forall\ \alpha \in S$ or equivalently $f' \in \mathrm{Int}(S,R  ) . $


  \end{proof}

  \medskip

   Similar to the case of $\Z,$ square-free elements can be defined in  any unique factorization domain $R$. When $d$ is a square-free element of $R,$ Lemma \ref{ivpcr} can be improved as follows. 
    \begin{lemma}\label{l1}  Let $f = \tfrac{g}{d} \in \K[x] $ be  a polynomial  of degree $k$ and $a_0, a_1, \ldots, a_k$ be a $d$-sequence, where $d$ 
     is a square-free  element of $R$. Let $s_{\pi}$ be the number of elements of $ S \subset R$ which are not congruent to each other (modulo $\pi$) for an irreducible $\pi$, then
    $$f \in \mathrm{Int}(S,R  )   \Leftrightarrow f(a_i) \in R\ \forall\ 0 \leq i \leq \min(s_{\pi} , k),$$
  for every divisor $\pi$ of $d$.
  
  
  \end{lemma}

  \begin{proof} By Lemma \ref{ivpcr} we have
  
   $$f \in \mathrm{Int}(S,R  )   \Leftrightarrow f(a_i) \in R\ \forall\ 0 \leq i \leq k.$$
  
If $s_{\pi}>k$ for a given irreducible $\pi$ dividing $d$, then we are done. Hence we assume $s_{\pi}<k.$ Let $S_{\pi}$ be the set of elements of $ S  $ which are not congruent to each other (modulo $\pi$) for a given irreducible $\pi$. Then the elements of $S_{\pi}$ form a $\pi$-sequence of $S$ in any order. Hence, if $a_0, a_1, \ldots, a_k$ is a $d$-sequence, then the first $s_{\pi}$ elements of this sequence are congruent to a unique element of $S_{\pi},$ where $s_{\pi}$ is the cardinality of the set $S_{\pi}$ and $d$ is a multiple of $\pi$. Observe that for any polynomial $h(x) \in R[x]$

  \begin{equation*}
  \begin{split}
  \tfrac{h}{\pi} \in  \mathrm{Int}(S,R  )   & \Leftrightarrow  \tfrac{h}{\pi}(S_{\pi}) \subset R,\\
 & \Leftrightarrow  \tfrac{h}{\pi}(a_i) \in R\ \forall\ 0 \leq i \leq s_{\pi}.
  \end{split}
  \end{equation*}

Also, if $ \pi$ and $ \pi'$ are two different irreducibles, then

$$\tfrac{h}{\pi}\ \mathrm{ and }\ \tfrac{h}{\pi'}\ \mathrm{ are\ members\ of  }\    \mathrm{Int}(S,R  ) \Leftrightarrow \tfrac{h}{\pi \pi'} \in  \mathrm{Int}(S,R  ) . $$

 In particular, this argument can be applied to the polynomial    $  \tfrac{g}{d} .$ This  completes  the proof.

  \end{proof}
  
  Sometimes this lemma may reduce  so much calculation. For instance, see Ex. (\ref{ex}). Now we prove our main theorem.

  \begin{theorem}\label{mainth}  Let $f = \tfrac{g}{d} \in \mathrm{Int}(S,R  ) $ be  a polynomial  of degree $k$ and $a_0, a_1, \ldots, a_k$ be a $d$-sequence. Then $f$ is irreducible iff the following holds:

for any factorization $g=g_1g_2$    and    a  divisor $\pi$  of $d$ such that $e_k$ is the maximum integer satisfying  $\pi^{e_k} \mid g_1(a_i)\ \forall\ 0 \leq i \leq \deg(g_1) , $   there exists an integer $j$  satisfying  $ 0 \leq j \leq \deg(g_2)  $ and    ${w_{\pi}(\tfrac{d}{\pi^{e_k}})}   \nmid  g_2(a_j).$
  
  \end{theorem}

 \begin{proof} For a given  polynomial $f = \tfrac{g}{d} \in \mathrm{Int}(S,R  ) $, suppose for every factorization    $g=g_1g_2$ 
 there exists a divisor  $\pi$  of $d$ satisfying $\pi^{e_k} \mid g_1(a_i)\ \forall\ 0 \leq i \leq \deg(g_1) $  and  $   \pi^{w_{\pi}(d)-e_k}   \nmid  g_2(a_j)$ for some non-negative integer $  j \leq \deg(g_2).$ Let us assume contrary that $f$ is reducible. Hence, there exists a factorization 
 
 $$f= \dfrac{h_1}{d_1}  \dfrac{h_2}{d_2}, $$
 
 such that $  \tfrac{h_1}{d_1}$ and $ \tfrac{h_2}{d_2}$ are members of $\mathrm{Int}(S,R  ). $ If for a divisor $\pi$ of $d$, $ w_{\pi}(d_1)=\pi^{e_k}= w_{\pi}(d) ,$ then this is contradiction to the assumption since $\pi^{0} \mid h_2(a)\ \forall\ a \in R. $ Similarly, $ w_{\pi}(d_1)$ cannot be $\pi^{0}.$  Hence we assume that $w_{\pi}(d_1)$ is a proper divisor of $ w_{\pi}(d) .$ In this case by assumption  $   \pi^{w_{\pi}(d)-e_k}   \nmid  h_2(a_j)$ for some positive integer $j$ satisfying $     j \leq \deg(h_2).$ By Lemma \ref{ivpcr} it follows that  $\tfrac{h_2}{d_2} $ cannot be a member of $\mathrm{Int}(S,R  ), $ which is again a contradiction. Hence, the polynomial must be irreducible.

 Now we assume that  $f = \tfrac{g}{d} \in \mathrm{Int}(S,R  ) $ is irreducible. For any factorization  $g=g_1g_2$ we can find suitable $d_1$ and $d_2$ such that 
 $$f= \dfrac{h_1}{d_1}  \dfrac{h_2}{d_2}, $$
 
where $  \tfrac{h_1}{d_1}$ is a member of $\mathrm{Int}(S,R )$ and $    \tfrac{h_2}{d_2}$ is not.   Since $    \tfrac{h_2}{d_2}$  does not belong to $ \mathrm{Int}(S,R ),$ hence by Lemma \ref{ivpcr} there exists a divisor $\pi$  of $d_2$, such that $ w_{\pi}(d_2)$ does not divide $ h_2(a_i)$ for some $ 0 \leq i \leq \deg(h_2) .$ Clearly  $ w_{\pi}( \tfrac{d}{d_2} )$ divides $ h_1(a_j)\ \forall\ 0 \leq j \leq \deg(h_1) $ completing the proof.


 \end{proof}

    We give some examples to illustrate our theorem.
    
    \begin{Ex} Let us check the   irreducibility of the polynomial $$ f=\frac{1}{9}(x^6+4x^5-8x^4+20x^3+4x^2+24x+27)  $$ in $\mathrm{Int}(\Z  ). $   In this case we have only the following way of factorization 
    
    $$ f=  \frac{1}{9}   ( x^3-2x^2+2x+3  ) (x^3+6x^2+2x+9)  .$$

  Here the degree of both  polynomials is three and we know that $ 0,1,2,3$ is a $d$-sequence for any integer $d$ of length three. Hence, we check the values of one polynomial at these points. Let $f_1= x^3+6x^2+2x+9,$ then one is the maximum positive integer such that $3^1 \mid f_1(i)$ for $i= 0,1,2,3.$ Taking  $f_2=x^3-2x^2+2x+3$, it can be seen that $3^{2-1}$ does not divide  $f_2(1)= 4$. Hence, the polynomial is irreducible.

      \end{Ex}  
      
      Sometimes if the factorization is known then it may be possible to predict the irreducibility of an integer-valued polynomial with a minuscule amount of information. We give an example to illustrate this.

    \begin{Ex}\label{ex} Let us test the irreducibility of the polynomial

    $$f=\frac{1}{6}  (x^{10}-22x^9+205x^8-1049x^7+3195x^6-5865x^5+6247x^4-
    3456x^3+720x^2$$ $$+18x-6) 
    $$
          in Int$(\Z)$ with $g_1= x^5-11x^4+42x^3-62x^2+26x-2 $ and  $g_2= x^5-11x^4+42x^3-63x^2+30x+3$ as known polynomials such that  $f=\tfrac{g_1g_2}{6}.$
       
    We can start with the prime $3$ since $g_1(0)=2$ is not a multiple of three. In view of Lemma \ref{l1}, we need to find a non-negative integer $0 \leq j \leq 2$ (not $0 \leq  j \leq 5$  ), such that  $ 3 \nmid    g_2(j).$ Observe that $ 3 \nmid    g_2(1)$, hence $f$ is irreducible.
    
    \end{Ex}

    In the above example we determined the irreducibility by merely checking at three points. This could be the case even if the degree of polynomial is very high. In such cases our method becomes very easy and practical.


  \section{Further generalizations}\label{secirrgen}

In the previous sections we tested the irreducibility of a polynomial $f = \tfrac{g}{d} \in \mathrm{Int}(S,D  ) $  by  using the unique factorization of the element $d \in R$. In this section, we suggest a generalization of the Theorem \ref{mainth} for some special domains. Assume the ideal generated by $d$ in a domain $D$ factors uniquely as a product of prime ideals. Then we can use the similar reasoning to get a $d$-sequence in this setting as well.  Recall  that each ideal  in a Dedekind domain factors uniquely as  a product of     prime ideals. Hence, we can generalize Theorem \ref{mainth}  to an arbitrary subset $S$ of a Dedekind domain $D$. For the sake of completeness we state the result 
(whose rigorous proof will be supplied in one of the subsequent articles) which can be proved by using  essentially the same technique.

 \begin{theorem}  Let $S$ be  an arbitrary subset of a Dedekind domain $D$ and $f = \tfrac{g}{d} \in \mathrm{Int}(S,D  ) $ be  a polynomial  of degree $k.$ If $a_0, a_1, \ldots, a_k$ is a $d$-sequence, then $f$ is irreducible iff the following holds:

for any factorization $g=g_1g_2$    and    a  prime ideal $P$  dividing $d$ such that $e_k$ is the maximum integer satisfying  $P^{e_k} \mid g_1(a_i)\ \forall\ 0 \leq i \leq \deg(g_1) , $   there exists an integer $j$  satisfying  $ 0 \leq j \leq \deg(g_2)  $ and    ${w_{P}(\tfrac{d}{P^{e_k}})}   \nmid  g_2(a_j).$
  
  \end{theorem}

\medskip

For a given subset $S$ of a domain $D$,  Theorem \ref{mainth} can be applied very easily to test the irreducibility of any   polynomial $f = \tfrac{g}{d} \in \mathrm{Int}(S,D  )  $, if $d$ has a unique factorization (into irreducibles or prime ideals).  For instance, when $D$ is a quotient of a Dedekind domain, we can rely on the Theorem \ref{mainth}. 
In the case when the ideal generated by $d$ factors uniquely into prime ideals, we must assume that the underlying domain is Noetherian since we can factor $d$ as a finite product of irreducibles in this case.  
In conclusion, when `$d$' has a unique factorization, Theorem \ref{mainth} remains valid for an arbitrary subset $S$ of a domain $D$. For instance, we have the following corollary

\begin{corollary}\label{cor} Let  $S$ be an arbitrary subset  of a domain $D$ and $f = \tfrac{g}{d} \in \mathrm{Int}(S,D  ) $ be  a polynomial  of degree $k$ where $d$ is an irreducible element. Let  $a_0, a_1, \ldots, a_k$ be a $d$-sequence,  then $f$ is irreducible iff the following holds:

for any factorization $g=g_1g_2$  there exist  integers $i,j$  satisfying  $ 0 \leq i \leq \deg(g_1)$    and  $ 0 \leq j \leq \deg(g_2)$ such that  $w_{d}(g_1(a_i)) =w_{d}(g_2(a_j))=d^0.$

\end{corollary}


In practice, whenever $d$ is irreducible (or `square-free') Lemma \ref{l1} can be used to get more sharper results. Hence, results similar to Corollary \ref{cor} can be improved further.

\medskip

In conclusion, we would like to emphasize that this article is an initial step to test the irreducibility of  integer-valued polynomials over   arbitrary subsets of a domain. We believe that  concepts similar to $d$-sequences  would be helpful in testing the   irreducibility of  integer-valued polynomials over arbitrary subsets of a domain in near future.  This seems   a very fertile area of research, which has not been explored   so far.

    \section*{Acknowledgement}  We   thank Dr. A. Satyanarayana Reddy and Dr. Krishnan Rajkumar for their helpful suggestions.   We also thank the referee for a careful reading of the manuscript and giving several
suggestions to correct misprints and improve readability.


\begin{thebibliography}{10}

\bibitem{snakirr}
Austin Antoniou, Sarah Nakato, and Roswitha Rissner.
\newblock Irreducible polynomials in {I}nt$(\mathbb{Z})$.
\newblock In {\em ITM Web of Conferences}, volume~20, page 01004. EDP Sciences,
  2018.

\bibitem{Bhar1}
Manjul Bhargava.
\newblock {$P$}-orderings and polynomial functions on arbitrary subsets of
  {D}edekind rings.
\newblock {\em J. Reine Angew. Math.}, 490:101--127, 1997.

\bibitem{Cahen}
Paul-Jean Cahen and Jean-Luc Chabert.
\newblock {\em Integer-valued polynomials}, volume~48 of {\em Mathematical
  Surveys and Monographs}.
\newblock American Mathematical Society, Providence, RI, 1997.

\bibitem{Whatyoushouldknowaboutintegervaluedpolynomials}
Paul-Jean Cahen and Jean-Luc Chabert.
\newblock What you should know about integer-valued polynomials.
\newblock {\em Amer. Math. Monthly}, 123(4):311--337, 2016.

\bibitem{Chabert}
Jean-Luc Chabert.
\newblock Integer-valued polynomials: looking for regular bases (a survey).
\newblock In {\em Commutative algebra}, pages 83--111. Springer, New York,
  2014.

\bibitem{Cahenfactorial}
Jean-Luc Chabert and Paul-Jean Cahen.
\newblock Old problems and new questions around integer-valued polynomials and
  factorial sequences.
\newblock In {\em Multiplicative ideal theory in commutative algebra}, pages
  89--108. Springer, New York, 2006.

\bibitem{Chapman}
Scott~T. Chapman and Barbara~A. McClain.
\newblock Irreducible polynomials and full elasticity in rings of
  integer-valued polynomials.
\newblock {\em J. Algebra}, 293(2):595--610, 2005.

\bibitem{fdnnfrisch2}
Sophie Frisch.
\newblock Integer-valued polynomials on algebras: a survey.
\newblock {\em Actes des rencontres du CIRM}, 2(2):27--32, 2010.

\bibitem{Peruginelli'sapplication}
Giulio Peruginelli.
\newblock Factorization of integer-valued polynomials with square-free
  denominator.
\newblock {\em Comm. Algebra}, 43(1):197--211, 2015.

\bibitem{Devendrafixed}
Devendra Prasad.
\newblock {\em Fixed Divisors and Generalized Factorials}.
\newblock PhD thesis, Shiv Nadar University, Greater Noida, 2019.

\bibitem{Prasad2019}
Devendra Prasad.
\newblock A generalization of selfridge's question.
\newblock {\em arXiv preprint}, 2019.

\bibitem{prasadsurvey}
Devendra Prasad, Krishnan Rajkumar, and A~Satyanarayana Reddy.
\newblock A survey on fixed divisors.
\newblock {\em Confluentes Mathematici}, 11(1):29--52, 2019.

\bibitem{Devendra}
Krishnan Rajkumar, A~Satyanarayana Reddy, and Devendra~Prasad Semwal.
\newblock Fixed divisor of a multivariate polynomial and generalized factorials
  in several variables.
\newblock {\em J. Korean Math. Soc.}, 55(6):1305--1320, 2018.

\bibitem{vajaituapp}
Marian V\^aj\^aitu.
\newblock An inequality involving the degree of an algebraic set.
\newblock {\em Rev. Roumaine Math. Pures Appl.}, 43(3-4):451--455, 1998.

\bibitem{vajaituapp1}
Marian V\^aj\^aitu and Alexandru Zaharescu.
\newblock A finiteness theorem for a class of exponential congruences.
\newblock {\em Proc. Amer. Math. Soc.}, 127(8):2225--2232, 1999.

\bibitem{fdnnwerner3}
Nicholas~J. Werner.
\newblock Integer-valued polynomials on algebras: a survey of recent results
  and open questions.
\newblock In {\em Rings, polynomials, and modules}, pages 353--375. Springer,
  Cham, 2017.

\end{thebibliography}
\end{document}